\documentclass[12pt]{amsart} \usepackage{amsmath,bm}
\usepackage{amssymb} \usepackage{amsthm}
\usepackage[hidelinks]{hyperref}

\newcommand{\Cay}{\operatorname{Cay}}

\newcommand{\mb}{\mathbf}

\newtheorem{theorem}{Theorem}
\newtheorem{lemma}[theorem]{Lemma}
\newtheorem{proposition}[theorem]{Proposition}

\newtheorem{conjecture}[theorem]{Conjecture}

\numberwithin{theorem}{section}
\numberwithin{equation}{section}

\theoremstyle{definition}

\newtheorem{observation}[theorem]{Observation}
\newtheorem{definition}[theorem]{Definition}

\newtheorem{problem}[theorem]{Problem}
\newtheorem{remark}[theorem]{Remark}

\title[Katznelson's problem]{Special cases and equivalent forms of Katznelson's problem on recurrence}

\author{John T. Griesmer}
\email{jtgriesmer@gmail.com}
\address{Department of Applied Mathematics and Statistics, Colorado School of Mines, Golden, Colorado}
\usepackage{amstext}

\begin{document}
\begin{abstract}
  We make the following three observations regarding a question popularized by Katznelson: is every subset of $\mathbb Z$ which is a set of Bohr recurrence  also a set of topological recurrence?
  \begin{enumerate}

  	\item[(i)]  If $G$ is a countable abelian group and $E\subseteq  G$ is an $I_0$ set, then every subset of $E-E$ which is a set of Bohr recurrence is also a set of topological recurrence.  In particular every subset of $\{2^n-2^m : n,m\in \mathbb N\}$ which is a set of Bohr recurrence is a set of topological recurrence.

	\item[(ii)]  Let $\mathbb Z^{\omega}$ be the direct sum of countably many copies of $\mathbb Z$ with standard basis $E$. If every subset of $(E-E)-(E-E)$ which is a set of Bohr recurrence is also a set of topological recurrence, then every subset of every countable abelian group which is a set of Bohr recurrence is also a set of topological recurrence.

	\item[(iii)]  Fix a prime $p$ and let $\mathbb F_p^\omega$ be the direct sum of countably many copies of $\mathbb Z/p\mathbb Z$ with basis $(\mb e_i)_{i\in \mathbb N}$. If for every $p$-uniform hypergraph with vertex set $\mathbb N$ and edge set $\mathcal F$ having infinite chromatic number, the Cayley graph on $\mathbb F_p^\omega$ determined by $\{\sum_{i\in F}\mb e_i:F\in \mathcal F\}$ has infinite chromatic number, then every subset of $\mathbb F_p^\omega$ which is a set of Bohr recurrence is a set of topological recurrence.
\end{enumerate}

\end{abstract}

	\maketitle

\section{Recurrence properties}\label{sec:Intro}
	We consider a problem raised in Section 4 of \cite{Katznelson}, stated as Conjecture \ref{conj:KatznelsonGeneral} below; see \cite{GKR} for history and overview of this problem.  Our main results are Proposition \ref{prop:RelativeKatznelson2Diff}, a very special case of Conjecture \ref{conj:KatznelsonGeneral}, and Proposition \ref{prop:RelativeKatznelson4Diff}, which reduces Conjecture \ref{conj:KatznelsonGeneral} to a special case.

\subsection{Definitions and main problem}  We write $\mathbb R$ for the group of real numbers with the usual topology, $\mathbb Z$ for the group of integers, and $\mathbb T$ for $\mathbb R/\mathbb Z$ with the quotient topology. The usual translation invariant metric on $\mathbb T^d$ is defined as follows: for $\mb x = (x_1,\dots,x_d)\in \mathbb T^d$, we write $\|\mb x\|$ for $\max_{j\leq d} \min_{n\in \mathbb Z} |x_j-n|$ (adopting the usual abuses of notation).  This makes $(\mb x,\mb y)\mapsto \|\mb x  - \mb y\|$ a translation invariant metric on $\mathbb T^d$.

We write $\mathcal S^1$ for the circle group $\{z\in \mathbb C:|z|=1\}$ with the operation of multiplication, which is  isomorphic as a topological group to $\mathbb T$, under the identification $t\in \mathbb T \leftrightarrow \exp(2\pi i t)\in \mathcal S^1$.

Let $G$ be a topological abelian group.  A \emph{character} of $G$ is a continuous homomorphism $\chi:G\to \mathcal S^1$. In this paper we consider only discrete groups, where every homomorphism to $\mathcal S^1$ is a character. A \emph{trigonometric polynomial} $p:G\to \mathbb C$ is a linear combination of characters.

If $A, B\subseteq  G$, we write $A+B$ for $\{a+b: a\in A, b\in B\}$ and $A-A$ for $\{a-a':a, a'\in A\}$.  The \emph{upper Banach density} of $A\subseteq  G$ is
\[d^*(A):=\sup\{\lambda(1_A):\lambda \text{ is an invariant mean on $l^{\infty}(G)$}\}.\] See \cite{BergelsonMcCutcheon} or \cite{BBF} for general exposition.  We do not prove results concerning upper Banach density in this article, but we will refer to such results from the literature.

\begin{definition}\label{def:Recurrence} Let $G$ be a discrete abelian group. If $S\subseteq  G$, we say that $S$ is a \emph{set of}
\begin{enumerate}

\item[$\bullet$] \emph{density recurrence} if for every $A\subseteq  G$ with $d^*(A)>0$, we have $(A-A)\cap S\neq \varnothing$.
  \item[$\bullet$] \emph{chromatic recurrence} if for every $r\in \mathbb N$ and every  cover of $G$ by $r$ sets $A_1,\dots, A_r$, there is a $j\leq r$ such that $(A_j-A_j)\cap S\neq \varnothing$.

      Equivalently, $S$ is a set of chromatic recurrence if for every $r\in \mathbb N$ and every $f:G\to \{1,\dots,r\}$, there exists $a, b\in G$ such that $f(a)=f(b)$ and $b-a\in S$.

  \item[$\bullet$] \emph{Bohr recurrence} if for every $d\in \mathbb N$, every homomorphism $\rho:G\to \mathbb T^d$, and every $\varepsilon>0$, there exists $s\in S$ such that $\|\rho(s)\|<\varepsilon$.

      Equivalently, $S$ is a set of Bohr recurrence if for every finite set $\{\chi_1,\dots, \chi_d\}$ of characters of $G$ and all $\varepsilon>0$, there exists $s\in S$ such that $\max_{j}|\chi_j(s)-1|<\varepsilon$.
\end{enumerate}
\end{definition}

The equivalence asserted in the third definition is due to the isomorphism of $\mathbb T$ with $\mathcal S^1$ and the fact that a collection of $d$ characters corresponds to a homomorphism $(\chi_1,\dots,\chi_d):G\to (\mathcal S^1)^d$.

\begin{definition}\label{def:BohrNhood}
  We say that $B\subseteq  G$ is a \emph{Bohr neighborhood of $0$} if there exists $d\in \mathbb N$, $\varepsilon>0$, a homomorphism $\rho:G\to \mathbb T^d$, and a neighborhood $U$ of $0$ in the usual topology on $\mathbb T^d$ such that $B$ contains $\rho^{-1}(U)$.  A Bohr neighborhood of $g\in G$ is a set of the form $U+g$, where $U$ is a Bohr neighborhood of $0$.  The \emph{Bohr topology} on $G$ is the smallest topology containing all Bohr neighborhoods.
\end{definition}
Note that $S\subseteq  G$ is a set of Bohr recurrence if and only if $S\cap B\neq \varnothing$ for every Bohr neighborhood $B$ of $0$ in $G$.

It is easy to verify that every set of density recurrence is a set of chromatic recurrence, and that every set of chromatic recurrence is a set of Bohr recurrence. Kriz \cite{Kriz} constructed a subset of $\mathbb Z$ which is a set of chromatic recurrence and not a set of density recurrence.

The question of whether every set of Bohr recurrence in $\mathbb Z$ is a set of chromatic recurrence was suggested by Veech \cite{Veech} (and implicitly by F{\o}lner \cite{Folner}), and popularized by Katznelson in \cite{Katznelson}; no conjecture was made in \cite{Katznelson} as to the correct answer.  While \cite{Katznelson} considers only the group $\mathbb Z$, the analogous question for arbitrary abelian groups is implicit in the surrounding literature.
There is no countably infinite abelian group $G$ where the answer is known.  Although we reserve judgment as to which answer is correct, stating the question as a conjecture makes the ensuing discussion more natural.
\begin{conjecture}\label{conj:KatznelsonGeneral}
  Let $G$ be a countably infinite abelian group.  If $S\subseteq  G$ is a set of Bohr recurrence, then $S$ is a set of chromatic recurrence.
\end{conjecture}

\subsection{Measurable recurrence and topological recurrence}

In abelian groups, the sets of chromatic recurrence are precisely the sets of topological recurrence.  Likewise, sets of density recurrence are sets of measurable recurrence and vice versa. Since this article never addresses dynamical systems per se, we use the terminology which is more closely aligned with our definitions. See \cite{BergelsonMcCutcheon} for a general discussion of the equivalences between various recurrence properties.

\subsection{Outline}  Our first result is Proposition \ref{prop:RelativeKatznelson2Diff}, which says that if $E$ belongs to a certain class of subsets of $G$, then every subset $S\subseteq  E-E$ which is a set of Bohr recurrence is also a set of chromatic recurrence.  A special case of this proposition for $G=\mathbb Z$ says that if $S\subseteq  \{2^n-2^m: m, n\in \mathbb N\}$ is a set of Bohr recurrence, then $S$ is a set of chromatic recurrence. This is in contrast to Kriz's construction \cite{Kriz}, which proves that if $E\subseteq  \mathbb Z$ is infinite, then there is a set $S\subseteq  E-E$ which is a set of chromatic recurrence and not a set of density recurrence.  Kriz \cite{Kriz} only stated that there is a subset of $\mathbb Z$ which is a set of chromatic recurrence and not of density recurrence, but minor modifications yield the more general result. The article \cite{KrizInDiff} proves the generalization explicitly.

In \S\ref{sec:Reduction} we state Conjecture \ref{conj:SpecialKatznelson4Z}, the special case of Conjecture \ref{conj:KatznelsonGeneral} where $G$ is the direct sum of countably many copies of $\mathbb Z$ and $S$ is contained in $(\mathcal E_1-\mathcal E_1)-(\mathcal E_1-\mathcal E_1)$, where $\mathcal E_1$ is the standard basis of $G$. Proposition \ref{prop:RelativeKatznelson4Diff} shows that this special case implies the full conjecture.  Conjecture \ref{conj:SpecialKatznelson4Z} is superficially similar to Proposition \ref{prop:RelativeKatznelson2Diff}, but we show in \S\ref{sec:DashHopes} that it is not susceptible to the same proof.

In \S\ref{sec:SpecialCases} we consider the special case of Conjecture \ref{conj:KatznelsonGeneral} where the ambient group is a vector space over a finite field of odd characteristic.  Proposition \ref{prop:KEpImpliesKFp} and Lemma \ref{lem:EpContinuity} show that in this setting, Conjecture \ref{conj:KatznelsonGeneral} can be reduced to the study of an easy-to-describe class of Bohr recurrent sets, first considered in \cite{GivensThesis} and \cite{GivensKunen} as a means to distinguish Bohr topologies on various groups.

Two problems suggested by Propositions \ref{prop:RelativeKatznelson2Diff} and \ref{prop:RelativeKatznelson4Diff} are stated in \S\ref{sec:Suggestions}.  In \S\ref{sec:Folner} we explain how Theorem \ref{thm:Bogoliouboff}, a seemingly stronger version of Bogoliouboff and F{\o}lner's theorem on iterated difference sets, actually follows from F{\o}lner's original proof.

\subsection{Previous work}   Conjecture \ref{conj:KatznelsonGeneral} has been studied  in \cite{GlasnerMinimal}, \cite{BoshGlasner},   \cite{GrivauxRoginskaya}, \cite{GKR}, and \cite{HostKraMaass}.  The exposition in \cite{GKR} thoroughly summarizes the history of the problem.

\subsection{Acknowledgements} We thank Anh Le for suggesting the proof in Remark \ref{rem:WeakI0}.  An anonymous referee contributed many corrections and improvements.

\section{Subsets of difference sets}\label{sec:SminusS}

If $G$ is an abelian group, we say $\psi:G\to \mathbb C$ is \emph{(uniformly) almost periodic} if $\psi$ is a uniform limit of trigonometric polynomials.

\begin{definition}\label{def:I0}
   If $G$ is an abelian group and $E\subseteq  G$, we say that $E$ is an \emph{$I_0$ set} if for every bounded  $f:E\to \mathbb C$, there is an almost periodic $\psi:G\to \mathbb C$ such that $\psi|_{E}=f$.
\end{definition}
See \cite{GrahamHareBook}, \cite{KunenRudin}, or \cite{LeInterpolation} for an overview of $I_0$ sets.  An early result \cite{Strzelecki} says that if  $A = \{a_1<a_2<\dots \}\subseteq  \mathbb Z$ is lacunary (meaning $\liminf_{n\to\infty} a_{n+1}/a_n>1$), then $A$ is an $I_0$ set.  Every infinite abelian group contains an $I_0$ set of the same cardinality as the group (see \cite{KunenRudin}).

\begin{definition}\label{def:Independent} A set $E\subseteq  G$ is \emph{independent} if for all $d\in\mathbb N$ and all $e_1,\dots,e_d\in E$, the only integer solutions $n_i$ to the equation $n_1e_1+\cdots +n_de_d=0$ satisfy $n_1e_1=\cdots = n_de_d=0$.
\end{definition}
  Independent sets are proved to be $I_0$ sets in Corollary 3.3 of \cite{KunenRudin}.  Examples of independent sets include linearly independent subsets of vector spaces.

\begin{proposition}\label{prop:RelativeKatznelson2Diff} Let $G$ be a countable abelian group and let $E\subseteq  G$ be an $I_0$ set.  If $S\subseteq  E-E$ and $S$ is a set of Bohr recurrence, then $S$ is a set of chromatic recurrence.
\end{proposition}

We prove Proposition \ref{prop:RelativeKatznelson2Diff} after  the following lemma.
\begin{lemma}\label{lem:BohrRec}
  Let $G$ be an abelian group.  The following are equivalent.
  \begin{enumerate}
    \item\label{item:BohrRecDef} $S$ is a set of Bohr recurrence.

    \item\label{item:SmallPoly}  For every almost periodic $\psi:G\to \mathbb C$ and all $\varepsilon>0$, there exists $s\in S$ such that $|\psi(g+s)-\psi(g)|<\varepsilon$ for all $g\in G$.
  \end{enumerate}
\end{lemma}
\begin{proof}
To prove that (\ref{item:BohrRecDef}) implies (\ref{item:SmallPoly}), it suffices to prove that (\ref{item:BohrRecDef}) implies the special case of (\ref{item:SmallPoly}) where $\psi$ is a trigonometric polynomial $p$, since every almost periodic function may be uniformly approximated by such $p$.  To prove this special case, assume $S$ is a set of Bohr recurrence, $p=\sum_{j=1}^d c_j\chi_j$ is a trigonometric polynomial on $G$, and $\varepsilon>0$.  Let $M=\max_{j\leq d}|c_j|$ and let $\varepsilon'=\varepsilon/(Md+1)$.  Since $S$ is a set of Bohr recurrence, we may choose $s\in S$ such that $|\chi_j(s)-1|<\varepsilon'$ for all $j\leq d$. Expanding $|p(g+s)-p(g)|$ we get
  \begin{align*}
    \Bigl|\sum_{j=1}^d c_j\chi_j(g+s)-c_j\chi_j(g)\Bigr|&\leq \sum_{j=1}^d |c_j||\chi_j(g+s)-\chi_j(g)|\\
    &= \sum_{j=1}^d |c_j||\chi_j(s)-1|\\
    &\leq  dM\varepsilon'\\
    &< \varepsilon.
  \end{align*}
We will not need the fact that (\ref{item:SmallPoly}) implies (\ref{item:BohrRecDef}), so we leave its proof as an exercise. \end{proof}

\begin{proof}[Proof of Proposition \ref{prop:RelativeKatznelson2Diff}]
  Fixing $G$, $E$, and $S$ in the hypothesis, it suffices to prove that for all $r\in \mathbb N$ and every $f:G\to \{1, \dots, r\}$, there exists $a, b\in G$  such that $f(b)=f(a)$ and $b-a \in S$.  So fix $r\in \mathbb N$ and $f:G\to \{1,\dots,r\}$.  Since $E$ is an $I_0$ set, we may choose an almost periodic $\psi:G\to \mathbb C$ such that $f(x)=\psi(x)$ for all $x\in E$.  By Lemma \ref{lem:BohrRec}, there is a $y\in S$ such that
  \begin{equation}\label{eqn:puniform}
  |\psi(x+y)-\psi(x)|<1 \quad \text{for all } x\in G.
   \end{equation}
   Since $S\subseteq  E-E$, we can write $y$ as $b-a$, where $a,b\in E$. Setting $x=a$ in (\ref{eqn:puniform}), we get $|\psi(b)-\psi(a)|<1$.  Then
  \[
  |f(b)-f(a)| = |\psi(b)-\psi(a)|<1.
   \]
   Since $f$ is integer valued, this implies $f(b)=f(a)$. Our choice of $a$ and $b$ implies $b-a\in S$, so this proves that $S$ is a set of chromatic recurrence.  \end{proof}

To see that Proposition \ref{prop:RelativeKatznelson2Diff} is not vacuous, we mention the following classical result, which is essentially the Poincar\'e recurrence theorem; see \cite{FurstenbergBook}, \cite{BergelsonMultifarious}, or \cite{FrantzMcCutcheon} for exposition.
\begin{theorem}\label{thm:Poincare} Let $G$ be a countable abelian group.  If $E$ is an infinite subset of $G$, then $\{a-b: a\neq b\in E\}$ is a set of density recurrence (and therefore a set of chromatic recurrence and a set of Bohr recurrence).
\end{theorem}

\begin{remark}
  If one could find an $I_0$ set $E$ such that $G=E-E$, then Conjecture \ref{conj:KatznelsonGeneral} would be a corollary of Proposition \ref{prop:RelativeKatznelson2Diff}.  But when $G$ has infinite cardinality, the difference set of an $I_0$ set is never all of $G$, as we prove in Lemma \ref{lem:I0minusI0}.
\end{remark}

The following definition will be useful in Remark \ref{rem:WeakI0} and Lemma \ref{lem:I0minusI0}; see \cite{RudinGroupsBook} for exposition.  %%TODO: refer to other instances of these concepts

\begin{definition}
The \emph{Bohr compactification} of a discrete abelian group $G$ is a compact abelian group $bG$ with an injective homomorphism $\iota: G\to bG$ such that $\iota(G)$ is topologically dense in $bG$, and every character $\chi\in \widehat{G}$ has the form $\chi'\circ \iota$ for some $\chi'\in \widehat{bG}$.
\end{definition}
By convention, we identify $G$ with its image $\iota(G)$ in $bG$, and say that every character of $G$ is the restriction of a character of $bG$.  We can thereby speak of the ``closure of $A$ in $bG$'' when $A\subseteq G$.
\begin{remark}\label{rem:WeakI0}
  The proof of Proposition \ref{prop:RelativeKatznelson2Diff} is easily modified to prove the analogous statement with the following ostensibly weaker hypothesis on $E$:
    \begin{center} ($*$) \begin{minipage}{0.9\textwidth} Every bounded $f:E\to \mathbb C$ can be uniformly approximated by a trigonometric polynomial restricted to $E$.\end{minipage}
      \end{center} The hypothesis ($*$) implies that $E$ is an $I_0$ set, so we would obtain no greater generality by using it in Proposition \ref{prop:RelativeKatznelson2Diff}.  To prove this implication, note that ($*$) implies that for any two disjoint subsets $E', E''\subseteq  E$ and all $\varepsilon>0$, there is a trigonometric polynomial $p$ with $|p(x)-1|<\varepsilon$ for all $x\in E'$ and $|p(x)-0|<\varepsilon$ for all $x\in E''$. This means that $E'$ and $E''$ have disjoint closures in the Bohr compactification of $G$. Since this applies to arbitrary disjoint subsets $E'$, $E'' \subseteq E$, Proposition 3.4.1 of \cite{GrahamHareBook} now implies $E$ is an $I_0$ set.
\end{remark}

\begin{remark}
 It is tempting to interpret Proposition \ref{prop:RelativeKatznelson2Diff} as saying ``When $E$ is an $I_0$ set, it is easy to prove that a subset of $E-E$ is a set of chromatic recurrence: just prove that it is Bohr recurrent.''  Upon reflection, it seems more correct to say that ``it is difficult to prove that a subset of $E-E$ is a set of Bohr recurrence: to do so one must prove that it is a set of chromatic recurrence.''
\end{remark}

One may wonder if Proposition \ref{prop:RelativeKatznelson2Diff} can be improved to conclude that $S$ is a set of density recurrence.  But it cannot, as the following proposition shows.
\begin{proposition}[\cite{KrizInDiff}, Theorem 1.2]\label{prop:KrizInDiff}
  Let $E\subseteq  \mathbb Z$ be infinite.  There is a set $S\subseteq  E-E$ such that $S$ is a set of chromatic recurrence and not a set of density recurrence.
\end{proposition}

In the hypothetical scenario where Conjecture \ref{conj:KatznelsonGeneral} is false, Propositions \ref{prop:RelativeKatznelson2Diff} and \ref{prop:KrizInDiff} demonstrate that the relationship between Bohr recurrence and chromatic recurrence differs from the relationship between chromatic recurrence and measurable recurrence.  Regardless of the status of Conjecture \ref{conj:KatznelsonGeneral}, the following problem seems interesting.
\begin{problem}\label{prob:ExtendI0}
  Extend Proposition \ref{prop:RelativeKatznelson2Diff} to a broader class of sets (broader than difference sets of finite unions of $I_0$ sets).
\end{problem}
The parenthetical remark is included because we believe that extending the result to finite unions of $I_0$ sets will be relatively straightforward, using facts from \cite{GrahamHareBook}, \cite{KunenRudin}, or \cite{LeInterpolation}.

\begin{remark}
  Proposition \ref{prop:RelativeKatznelson2Diff} is similar to Theorem 7.15 of \cite{KunenRudin}, which solves a special case of an open problem about Sidon sets in $\mathbb Z$, under the additional assumption that the set under consideration is contained in the difference set $E-E$ of a lacunary  set $E$.
\end{remark}

The following lemma shows that difference sets of $I_0$ sets are small, and in particular never equal to the entire ambient group.  If $S\subseteq  G$ and $g\in G$, we  say that $g$ is a \emph{Bohr limit point of $S$} if $(S\cap U)\setminus \{g\} \neq \varnothing$ for every Bohr neighborhood $U$ of $g$.  In other words, the closure $\overline{S\setminus \{g\}}$ in the Bohr topology contains $g$.

\begin{lemma}\label{lem:I0minusI0}  If $G$ is an abelian group and $E\subseteq  G$ is an $I_0$ set, then there is a Bohr neighborhood $U$ of $0$ in $G$ such that $0$ is the only Bohr limit point of $(E-E)\cap U$.
\end{lemma}

\begin{proof}
  Let $E\subseteq  G$ be an $I_0$ set.  By Theorems 5.3.1 and 5.3.9 of \cite{GrahamHareBook}, $E$ is a union of finitely many sets $E_1,\dots, E_r$ such that for each $j$, the only Bohr limit point of $E_j-E_j$ is $0$.  This property is inherited by subsets, so we may assume that the $E_j$ are mutually disjoint.  Let $\tilde{E}_j$ be the closure of $E_j$ in the Bohr compactification of $G$.  By Proposition 3.4.1 of \cite{GrahamHareBook}, the $\tilde{E}_j$ are mutually disjoint. So for each $i,j, i \neq j$, there is a Bohr neighborhood $U_{i,j}$ of $0$ such that $(E_i-E_j)\cap U_{i,j}=\varnothing$.  Letting $U$ be the intersection of these $U_{i,j}$, we get that $U\cap \bigcup_{i\neq j \leq r} (E_i-E_j)=\varnothing$.  Now $U\cap (E-E)\subseteq  \bigcup_{j=1}^{r} (E_j-E_j)$, and the only Bohr limit point of each $E_j-E_j$ is $0$.  So the only Bohr limit point of the union is $0$, as well.  Thus the only Bohr limit point of $U\cap (E-E)$ is $0$.
\end{proof}

\section{Iterated differences}\label{sec:Reduction}
In this section we state Conjecture \ref{conj:SpecialKatznelson4Z} and prove that it implies Conjecture \ref{conj:KatznelsonGeneral}.  The following theorem on Bohr neighborhoods (Defintion \ref{def:BohrNhood}) is a key ingredient in the proof and a primary impetus for investigations of difference sets.  Given a subset $A$ of an abelian group $G$, we write $\Delta_2(A)$ for the set $\{(a-b)-(c-d): a,b,c,d\in A \text{ are mutually distinct}\}$.

\begin{theorem}[Bogoliouboff \cite{Bogoliouboff}, F{\o}lner, {\cite[Theorem 1]{Folner}}]
\label{thm:Bogoliouboff}
 Let $G$ be a countably infinite abelian group.  If $G$ is written as a finite union of sets $A_1, \dots, A_r$, then for some $j\leq r$ the set $\Delta_2(A_j)$ is a Bohr neighborhood of $0$.
\end{theorem}
Readers familiar with the literature will note that we have omitted bounds on the rank and radius of the Bohr neighborhood in Theorem \ref{thm:Bogoliouboff}, which may be of interest in quantitative approaches to Conjecture \ref{conj:KatznelsonGeneral}.

Often Theorem \ref{thm:Bogoliouboff} is stated with $(A_j-A_j)-(A_j-A_j)$ in place of $\Delta_2(A_j)$, but imposing the additional restriction that the terms in the difference be distinct does not significantly alter the structure.  In \S \ref{sec:Folner} we explain how F{\o}lner's proof in \cite{Folner} yields Theorem \ref{thm:Bogoliouboff}.

Let $\mathbb Z^\omega$ denote the direct sum of countably many copies of $\mathbb Z$ with the usual presentation: elements of $\mathbb Z^\omega$ are sequences $(n_1,n_2,n_3,\dots)$ of integers where $n_j=0$ for all but finitely many $j$.  For each $j\in \mathbb N$, let $\mb e_j$ be the element of $\mathbb Z^{\omega}$ where $(\mb e_j)_j=1$ and $(\mb e_j)_k=0$ if $j\neq k$.   So $\mb e_1 = (1,0,0,\dots)$, $\mb e_2 = (0,1,0,0,\dots)$, etc.  We write $\mathcal E_1$ for $\{\mb e_j:j\in \mathbb N\}$.  The set $\Delta_2(\mathcal E_1)$ contains elements such as $(0,1,0,-1,1,-1,0, \dots)$.

\begin{conjecture}\label{conj:SpecialKatznelson4Z}
  If $S\subseteq  \Delta_2(\mathcal E_1)$ and $S$ is a set of Bohr recurrence in $\mathbb Z^\omega$, then $S$ is a set of chromatic recurrence.
\end{conjecture}

\begin{proposition}\label{prop:RelativeKatznelson4Diff}
Conjecture \ref{conj:SpecialKatznelson4Z} implies Conjecture \ref{conj:KatznelsonGeneral}.
\end{proposition}

\begin{remark}
Our interest in Conjecture \ref{conj:SpecialKatznelson4Z} arises from our failed attempts to prove that Proposition \ref{prop:RelativeKatznelson2Diff} implies Conjecture \ref{conj:KatznelsonGeneral}.  To analogize Proposition \ref{prop:RelativeKatznelson2Diff} with Conjecture \ref{conj:SpecialKatznelson4Z}, note that the former concerns subsets of a difference set of an $I_0$ set,  while the latter concerns subsets of an iterated difference set of $\mathcal E_1$, and $\mathcal E_1$ is an $I_0$ set (since it is independent).

One may hope for an easy proof of Conjecture \ref{conj:SpecialKatznelson4Z} along the lines of our proof of Proposition \ref{prop:RelativeKatznelson2Diff}.  We point out an obstruction to this hope in \S\ref{sec:DashHopes}.
\end{remark}

The proof of Proposition \ref{prop:RelativeKatznelson4Diff} consists of the following two lemmas.

\begin{lemma}\label{lem:LiftToE4}
  Let $H$ and $G$ be countably infinite abelian groups, let $Q\subseteq  H$, and let $\rho:H\to G$ be a homomorphism such that $\rho(Q)=G$.  If $S\subseteq  G$ is a set of Bohr recurrence, then $\rho^{-1}(S)\cap  \Delta_2(Q)$ is a set of Bohr recurrence in $H$.
\end{lemma}

\begin{proof}   Assuming $H,G,Q, S$, and $\rho$ are as in the hypothesis, we must prove that for all $d\in \mathbb N,$ all homomorphisms $\psi:H\to \mathbb T^d$, and all $\varepsilon>0$, there is an $h\in \rho^{-1}(S)\cap \Delta_2(Q)$ such that $\|\psi(h)\|<\varepsilon$.  To do so, we will find distinct elements $q_1,q_2,q_3,q_4\in Q$ such that
  \begin{equation}\label{eqn:Goal}
  \rho(q_1-q_2-q_3+q_4)\in S   \quad \text{and} \quad \|\psi(q_1-q_2-q_3+q_4)\|<\varepsilon.
  \end{equation}  So fix $d\in \mathbb N$,  a homomorphism $\psi:H\to \mathbb T^d$, and $\varepsilon>0$.  To find $q_i$ satisfying (\ref{eqn:Goal}), write $\mathbb T^d$ as a finite union of sets $B_1,\dots, B_r$ having diameter less than $\varepsilon/2$.  Then the preimages $\psi^{-1}(B_1),\dots,\psi^{-1}(B_r)$ cover $H$.  In particular, the sets $C_j:=\psi^{-1}(B_j)\cap Q$, $j\leq r$, cover $Q$.  The hypothesis that $\rho(Q)=G$ now implies that the images $A_1:=\rho(C_1)$, \dots, $A_r:=\rho(C_r)$ cover $G$.  Theorem \ref{thm:Bogoliouboff} then implies that for some $j$, the set $\Delta_2(A_j)$ is a Bohr neighborhood of $0_G$ in $G$.  Since we assumed $S$ is a set of Bohr recurrence, we then have $S\cap \Delta_2(A_j)\neq \varnothing$.  In particular, there are mutually distinct $a_1,a_2,a_3,a_4\in A_j$ and $s\in S$ such that $s = a_1-a_2-a_3+a_4$.  The definition of $A_j$ allows us to write the latter equation as
\begin{equation}\label{eqn:srhoe}
  s = \rho(q_1-q_2-q_3+q_4) \quad \text{where } q_i\in C_j \text{ for each } i.
\end{equation}
The definition of $C_j$ means that each $q_i$ lies in $Q$ and each $\psi(q_i)$ lies in $B_j$, which has diameter less than $\varepsilon/2$. This implies \begin{equation}\label{eqn:q4Triangle}\|\psi\bigl((q_1-q_2)-(q_3-q_4)\bigr)\|\leq \|\psi(q_1)-\psi(q_2)\|+\|\psi(q_3)-\psi(q_4)\| <\varepsilon.
\end{equation}  Now (\ref{eqn:srhoe}) and (\ref{eqn:q4Triangle}) together show that the $q_i$ satisfy (\ref{eqn:Goal}).
\end{proof}

\begin{lemma}\label{lem:Image}
  Let $H$ and $G$ be abelian groups and let $R\subseteq  H$ be a set of chromatic recurrence.  If $\rho:H\to G$ is a homomorphism, then $\rho(R)$ is a set of chromatic recurrence in $G$.
\end{lemma}

\begin{proof}
  Let $k\in \mathbb N$ and let $f:G\to \{1,\dots,k\}$.  Then $\tilde{f}:=f\circ \rho$ is a function from $H$ to $\{1,\dots,k\}$.  Since $R$ is a set of chromatic recurrence, there exist $h_1, h_2\in H$ such that $\tilde{f}(h_1)=\tilde{f}(h_2)$ and $h_2-h_1\in R$.  By the definition of $\rho$, we have $f(\rho(h_1))=f(\rho(h_2))$ and $\rho(h_2-h_1)\in \rho(R)$, meaning $\rho(h_2)-\rho(h_1)\in \rho(R)$.  Setting $a=\rho(h_1)$ and $b=\rho(h_2)$, we have found $a, b\in G$ such that $f(a)=f(b)$ and $b-a\in \rho(R)$.
\end{proof}

\begin{observation}\label{obs:Ind}
  If $G$ is an abelian group and $f:\mathcal E_1\to G$, then there is a homomorphism $\rho:\mathbb Z^\omega\to G$ such that $\rho|_{\mathcal E_1}=f$.  Such a $\rho$ is defined uniquely by $\rho(\sum c_j\mb e_j):=\sum c_jf(\mb e_j)$.
\end{observation}

\begin{proof}[Proof of Proposition \ref{prop:RelativeKatznelson4Diff}]
Let $G$ be a countably infinite abelian group, and let $S\subseteq  G$ be a set of Bohr recurrence.  By Observation \ref{obs:Ind} choose a homomorphism $\rho:\mathbb Z^\omega\to G$ such that $\rho(\mathcal E_1)=G$.  Assuming Conjecture \ref{conj:SpecialKatznelson4Z} holds, we will prove that $\rho^{-1}(S)$ is a set of chromatic recurrence.   To see this, first apply Lemma \ref{lem:LiftToE4} with $Q=\mathcal E_1$ to get that $\rho^{-1}(S)\cap \Delta_2(\mathcal E_1)$ is a set of Bohr recurrence.  Now Conjecture \ref{conj:SpecialKatznelson4Z} implies $\rho^{-1}(S)\cap \Delta_2(\mathcal E_1)$ is a set of chromatic recurrence, so $\rho^{-1}(S)$ is a set of chromatic recurrence, as well.   By Lemma \ref{lem:Image} we get that $S$ is a set of chromatic recurrence. \end{proof}

\begin{remark}
  Some special cases of Conjecture \ref{conj:SpecialKatznelson4Z} are easy to prove.  For example,  $\Delta_2(\mathcal E_1)$  contains an infinite difference set $E-E$, where $E:=\{\mb e_{2n}-\mb e_{2n+1} : n\in \mathbb N\}$.  This $E$ is an independent set (Definition \ref{def:Independent}), and therefore an $I_0$ set, in $\mathbb Z^{\omega}$. Proposition \ref{prop:RelativeKatznelson2Diff} then proves that every subset of $E-E$ which is a set of Bohr recurrence is also a set of chromatic recurrence.  We have not proved any case of Conjecture \ref{conj:SpecialKatznelson4Z} beyond such trivial instances.
\end{remark}

\begin{remark}\label{rem:Real}
  The only facts about sets of chromatic recurrence used in the proof of Proposition \ref{prop:RelativeKatznelson4Diff} are the following:
  \begin{enumerate}
   \item[$\bullet$] the homomorphic image of a set chromatic recurrence is again a set of chromatic recurrence (Lemma \ref{lem:Image}, which we applied only with a surjective $\rho$);
       \item[$\bullet$] the property of being a set of chromatic recurrence is \emph{upward closed}: if $S$ is a set of chromatic recurrence and $S\subset S'$, then $S'$ is also a set of chromatic recurrence.
   \end{enumerate}
         A trivial modification of the proof of Proposition \ref{prop:RelativeKatznelson4Diff} will therefore prove the following generalization, which may be of interest independent from Conjecture \ref{conj:KatznelsonGeneral}.
\end{remark}

\begin{proposition}\label{prop:Real}
  Let $G$ be a countable abelian group, let $\mathcal R$ and $\mathcal S$ be classes of subsets of  $\mathbb Z^\omega$ and $G$, respectively, such that
  \begin{enumerate}
    \item[(i)] if $\rho:\mathbb Z^{\omega}\to G$ is a surjective homomorphism and $R\in\mathcal R$, then $\rho(R)\in \mathcal S$;
    \item[(ii)] $\mathcal R$ and $\mathcal S$ are upward closed.
  \end{enumerate}
If every subset of $\Delta_2(\mathcal E_1)$ in $\mathbb Z^{\omega}$ which is a set of Bohr recurrence belongs to $\mathcal R$, then every subset of $G$ which is a set of Bohr recurrence belongs to $\mathcal S$.
\end{proposition}
Proposition \ref{prop:RelativeKatznelson4Diff} is recovered by considering the special case of Proposition \ref{prop:Real} where $\mathcal R$ is the collection of sets of chromatic recurrence in $\mathbb Z^\omega$ and $\mathcal S$ is the collection of sets of chromatic recurrence in an arbitrary countable abelian group $G$.

\section{Special cases in \texorpdfstring{$\mathbb F_p^\omega$}{Fpw}}\label{sec:SpecialCases}

Variants of Proposition \ref{prop:RelativeKatznelson4Diff} can be obtained from variants of Theorem \ref{thm:Bogoliouboff}; the resulting statements are especially appealing in vector spaces over finite fields.

Fix a prime $p$, let $\mathbb F_p$ be the additive group $\mathbb Z/p\mathbb Z$, and let $\mathbb F_p^\omega$ denote the direct sum of countably many copies of $\mathbb F_p$.  We adopt the usual presentation of $\mathbb F_p^\omega$: elements are written as $(x_1,x_2,x_3,\dots)$, where each $x_i\in \mathbb F_p$, and $\mb e_1 = (1,0,0,\dots,)$, $\mb e_2 = (0,1,0,\dots)$, $\dots$ forms the standard basis.

\begin{observation}\label{obs:BohrIsFiniteIndex} Since every homomorphism $\rho:\mathbb F_p^\omega \to \mathbb T^d$ takes values in the finite subgroup of $\mathbb T^d$ consisting of elements of order $p$, the Bohr neighborhoods of $0$ in $\mathbb F_p^\omega$ are exactly the sets containing a finite index subgroup of $\mathbb F_p^\omega$.  It follows that $S\subseteq  \mathbb F_p^\omega$ is a set of Bohr recurrence if and only if $S$ has nonempty intersection with every finite index subgroup of $\mathbb F_p^\omega$.
\end{observation}

For a finite subset $F\subseteq  \mathbb N$, let $\mb e_F:=\sum_{n\in F} \mb e_n$. Let $\mathcal E_{d}^{(p)}:=\{\mb e_F : F\subseteq  \mathbb N, |F| = d\}$.  The following conjecture is a variant of Conjecture \ref{conj:SpecialKatznelson4Z}.

\begin{conjecture}\label{conj:KatznelsonEp}
    Let $p$ be prime and let $d\in \mathbb N$ be divisible by $p$. Every subset of $\mathcal E_d^{(p)}$ which is a set of Bohr recurrence is also set of chromatic recurrence.
\end{conjecture}
When $p$ does not divide $d$, $\mathcal E_d^{(p)}$ is not a set of Bohr recurrence, since it has empty intersection with the index $p$ subgroup $\{x\in \mathbb F_p^\omega: \sum_{i\in \mathbb N} x_i = 0\}$.

Here is the special case of Conjecture \ref{conj:KatznelsonGeneral} for $\mathbb F_p^\omega$.

\begin{conjecture}\label{conj:KatznelsonFp}
  Every subset of $\mathbb F_p^\omega$ which is a set of Bohr recurrence is a set of chromatic recurrence.
\end{conjecture}

Clearly Conjecture \ref{conj:KatznelsonFp} implies Conjecture \ref{conj:KatznelsonEp}.  The purpose of this section is to prove the converse (assuming $d>2$), and to prove Lemma \ref{lem:EpContinuity}, an appealing characterization of the Bohr recurrent subsets of $\mathcal E_p^{(p)}$.

\begin{proposition}\label{prop:KEpImpliesKFp}
  Let $p$ be prime and let $d>2$ be divisible by $p$.  The special case of Conjecture \ref{conj:KatznelsonEp} for this $d$ and $p$ implies the special case of Conjecture \ref{conj:KatznelsonFp} for the same $p$.
\end{proposition}

\begin{remark}
  Conjecture \ref{conj:KatznelsonEp} is not vacuous: the set $\mathcal E_p^{(p)}$ itself is a set of Bohr recurrence, which can be verified directly, or viewed as a consequence of Lemma \ref{lem:LiftToEp}, or of Lemma \ref{lem:EpContinuity}.
\end{remark}

\begin{remark}
  We are currently unable to prove that the special case of Conjecture \ref{conj:KatznelsonEp} with $d=p=2$ implies Conjecture \ref{conj:KatznelsonFp}.  This special case of Conjecture \ref{conj:KatznelsonEp} is also a special case of Proposition \ref{prop:RelativeKatznelson2Diff}, since $\mathcal E_1^{(2)}$ is an $I_0$ set (as it is independent), and $\mathcal E_2^{(2)}\subseteq \mathcal E_1^{(2)}-\mathcal E_1^{(2)}$.
\end{remark}

We will use the following special case of \cite[Theorem 1.3]{GrBogoliouboff}.

\begin{lemma}\label{lem:BogFp}
  Let $p$ be prime and  let $d>2$ be divisible by $p$. If $A\subseteq  \mathbb F_p^\omega$ has $d^*(A)>0$,  then the $d$-fold sumset with distinct summands
  \[d\verb!^!A:= \{a_1+a_2+\cdots + a_d: a_i\in A \textup{ are mutually distinct}\}\]
   is a Bohr neighborhood of $0$.  Consequently, if $\mathbb F_p^\omega$ is covered by finitely many sets $A_1$,$\dots$, $A_r$, then for some $j$ the set $d\verb!^!A_j$ is a Bohr neighborhood of $0$.
\end{lemma}
\begin{remark} The statement of  \cite[Theorem 1.3]{GrBogoliouboff} uses $(d\verb!^!A) \cup \{0\}$ in place of $d\verb!^!A$ defined above.  However, the proof of \cite[Theorem 1.3]{GrBogoliouboff} implies that the ``$\cup\{0\}$'' can be omitted.  To see this, note that the proof of \cite[Lemma 1.8]{GrBogoliouboff} proves that the function $I$ defined therein satisfies $\{t:I(t)>0\}\subseteq \{c_1g_1+\cdots +c_dg_d : g_i\in G \text{ are mutually distinct} \}$, and that $\{t: I(t)>0\}$ contains a Bohr neighborhood of $0$.
\end{remark}

The proof of Proposition \ref{prop:KEpImpliesKFp} is similar to the proof of Proposition \ref{prop:RelativeKatznelson4Diff}.  It relies on the following variant of Lemma \ref{lem:LiftToE4}, whose proof uses Lemma \ref{lem:BogFp} in place of Theorem \ref{thm:Bogoliouboff}.
\begin{lemma}\label{lem:LiftToEp}
Let $p$ be prime, let $d>2$ be divisible by $p$, and let $\rho: \mathbb F_p^\omega\to \mathbb F_p^\omega$ be a homomorphism such that $\rho(\mathcal E_1^{(p)})=\mathbb F_p^\omega$.  If $S\subseteq  \mathbb F_p^\omega$ is a set of Bohr recurrence, then $\rho^{-1}(S)\cap \mathcal E_d^{(p)}$ is a set of Bohr recurrence, as well.
\end{lemma}

\begin{proof}
  Fix $p$, $d$, $\rho$, and $S$ as in the hypothesis.  We must prove that $S':=\rho^{-1}(S)\cap \mathcal E_d^{(p)}$ is a set of Bohr recurrence, which in this case simply means that $S'$ has nonempty intersection with every finite index subgroup of $\mathbb F_p^{\omega}$.  Fix such a subgroup $H$, and enumerate the cosets of $H$ as $H_1,\dots,H_r$.  Let $C_j:=H_j\cap \mathcal E_1^{(p)}$, so that $\mathcal E_1^{(p)}= C_1\cup \dots \cup C_r$.  Let $A_j:=\rho(C_j)$, $j=1,\dots, r$, so the hypothesis that $\rho(\mathcal E_1^{(p)})=\mathbb F_p^\omega$ implies that the $A_j$ cover $\mathbb F_p^\omega$.  Lemma \ref{lem:BogFp} now implies that for some $j\leq r$, $d\verb!^!A_j$ is a Bohr neighborhood of $0$.  Since $S$ is a set of Bohr recurrence we have that $S\cap d\verb!^!A_j\neq \varnothing$.  In particular, we can find an $s\in S$ and mutually distinct $a_1,\dots, a_d\in A_j$ such that $  s = a_1+\cdots + a_d$.  By the definition of $A_j$, this means there are mutually distinct $c_1,\dots,c_d\in C_j$ such that
\begin{equation}\label{eqn:srhoc}
  s=\rho(c_1)+\cdots + \rho(c_d).
\end{equation}
    By the definition of $C_j$, these $c_i$ all lie in the same coset of $H$, meaning $c_i+H=c_1+H$ for each $i$. Since $p$ divides $d$, we get that $c_1+\dots+c_d$ lies in $H$: to see this write $c_1+\dots + c_d+H = dc_1+H=H$.

  Now $c_i\in \mathcal E_1^{(p)}$ for each $i$, so $c:=c_1+\cdots + c_d\in \mathcal E_{d}^{(p)}$.  Equation (\ref{eqn:srhoc}) shows that $c\in \rho^{-1}(S)$, so we have proved that $\rho^{-1}(S)\cap \mathcal E_{d}^{(p)}\cap H\neq \varnothing$, as desired.
\end{proof}

\begin{proof}[Proof of Proposition \ref{prop:KEpImpliesKFp}]
Fix a prime $p$ and $d>2$ divisible by $p$.  Assuming Conjecture \ref{conj:KatznelsonEp} is true for this $p$ and $d$, we will prove that Conjecture \ref{conj:KatznelsonFp} is true for this $p$, as well. Fix a set of Bohr recurrence $S\subseteq  \mathbb F_p^\omega$.   Let $\rho:\mathbb F_p^\omega\to \mathbb F_p^\omega$ be a homomorphism such that $\rho(\mathcal E_1^{(p)})=\mathbb F_p^\omega$.  By Lemma \ref{lem:LiftToEp}, we have that $S':=\rho^{-1}(S)\cap \mathcal E_d^{(p)}$ is a set of Bohr recurrence.  Our assumption that Conjecture \ref{conj:KatznelsonEp} is true then implies that $S'$ is a set of chromatic recurrence.  Then Lemma \ref{lem:Image} implies $\rho(S')$ also a set of chromatic recurrence, and the containment $\rho(S')\subseteq  S$ implies $S$ is, as well.
\end{proof}

Now we characterize the sets of Bohr recurrence contained in $\mathcal E_p^{(p)}$.  We call a collection $\mathcal F$ of subsets of $\mathbb N$  \emph{partition regular}\footnote{This is an abuse of terminology: we should say that ``containing an element of $\mathcal F$ is a partition regular property of subsets of $\mathbb N$.''} if for every partition of $\mathbb N$ into finitely many cells, at least one of the cells contains an element of $\mathcal F$.

We write $\widehat{\mathbb F}_p^\omega$ for the set of characters of $\mathbb F_p^\omega$, so that $\widehat{\mathbb F}_p^\omega$ is a group under pointwise multiplication.  It is easy to verify that $\widehat{\mathbb F}_p^\omega$ is isomorphic to $\mathbb F_p^{\mathbb N}$, the product of countably many copies of $\mathbb F_p$.  To see this, note that every element $\xi \in \mathbb F_p^{\mathbb N}$ induces a character defined by $\chi(x):=e(\sum_{n\in \mathbb N} x_n\xi_n)$, where multiplication $x_n\xi_n$ is done in $\mathbb F_p$, and $e:\mathbb F_p\to \mathcal S^1$ is the homomorphism with $e(1)=e^{2\pi i/p}$.  Conversely, every character of $\mathbb F_p^\omega$ has this form: given $\chi\in \widehat{\mathbb F}_p^\omega$, define $\xi\in \mathbb F_p^{\mathbb N}$ to satisfy $e(\xi_n)=\chi(\mb e_n)$. Then $\chi(x)=e(\sum_{n\in \mathbb N} x_n\xi_n)$ for all $x\in \mathbb F_p^\omega$.

\begin{lemma}\label{lem:EpContinuity} Let $p$ be prime (not necessarily odd), let $\mathcal F$ be a collection of subsets of $\mathbb N$ each having cardinality $p$, and let $\mathcal E(\mathcal F):=\{\mb e_F:F\in \mathcal F\}\subseteq  \mathbb F_p^\omega$.  Then $\mathcal E(\mathcal F)$ is a set of Bohr recurrence if and only if $\mathcal F$ is partition regular.
\end{lemma}
The implication ``$\mathcal E(\mathcal F)$ is a set of Bohr recurrence if $\mathcal F$ is partition regular'' is a special case of Lemma 3.3 of \cite{GivensKunen} (originally proved in \cite{GivensThesis}).  The converse is a special case of Lemma 3.5 of \cite{GivensKunen}.

\begin{proof}
First suppose that $\mathcal F$ is a partition regular collection of subsets of $\mathbb N$ each having cardinality $p$. We will prove that $\mathcal E(\mathcal F)$ is a set of Bohr recurrence by showing that for every finite set of characters $\chi_1,\dots, \chi_r\in \widehat{\mathbb F}_p^\omega$, there is an $\mb e_F\in \mathcal E(\mathcal F)$ such that $\chi_j(\mb e_F)=1$ for each $j$.  So fix such a collection of characters $\chi_1,\dots, \chi_r$.

For $j\leq r$ and $m\leq p$, let $A_{j,m}:=\{n\in\mathbb N:\chi_{j}(\mb e_n)=\exp(2\pi im/p)\}$.  Let $\mathcal P=\{P_1,\dots,P_k\}$ be the partition of $\mathbb N$ induced by the $A_{j,m}$ (i.e.~the minimal nonempty elements of the algebra of subsets of $\mathbb N$ generated by the $A_{j,m}$). Then each $\chi_j$ is constant on the elements of $\mathcal P$.  By the partition regularity of $\mathcal F$, we may choose $F\in \mathcal F$ so that $F\subseteq  P$ for some element $P$ of $\mathcal P$. To see that $\chi_j(\mb e_F)=1$ for each $j$, note that $\chi_{j}$ is constant on $\{\mb e_n:n\in F\}$; call this constant $s_j$.  Then
\[\chi_j(\mb e_F)= \chi_j\Bigl(\sum_{n\in F} \mb e_n\Bigr)=\prod_{n\in F}\chi_j(\mb e_{n})=\prod_{n\in F} s_j =s_j^p=1,\]
since $\chi_j$ takes values in the $p^\text{th}$ roots of unity.

Now suppose that $\mathcal F$ is not partition regular, meaning there is a partition of $\mathbb N$ into finitely many cells $C_1, \dots, C_r$ such that for all $F\in \mathcal F$, $F$ is not contained in any of the $C_j$.  We will define homomorphisms $\psi_1,\dots,\psi_r:\mathbb F_p^\omega \to \mathbb F_p$ such that for each $F\in \mathcal F$, there is a $j$ with $\psi_j(\mb e_F)\neq 0$.  This will show that $\mathcal E(\mathcal F)$ has empty intersection with the finite index subgroup $\bigcap_{j=1}^d \ker \psi_j$, meaning $\mathcal E(\mathcal F)$ is not a set of Bohr recurrence.

  For $j\leq r$ and $x\in \mathbb F_p^\omega$, define $\psi_j(x):=\sum_{n\in C_j}x_n$.  This is well defined, since $x_n=0$ for all but finitely many $n$.  Let $F\in \mathcal F$.  Then $\psi_j(\mb e_F)=|F\cap C_j|$ mod $p$.  Since $F$ is not contained in any of the $C_j$, there is a $j$ such that $0<|F\cap C_j|<p$, meaning $\psi_j(\mb e_F)\neq 0$ for this $j$. This proves that $\mathcal E(\mathcal F)$ has empty intersection with $\bigcap_{j=1}^d \ker \psi_j$. \end{proof}

\begin{remark}
  Specializing Lemma \ref{lem:EpContinuity} to $p=2$, the set $\mathcal E(\mathcal F)$ becomes a subset of $\mathcal E_1^{(2)}-\mathcal E_1^{(2)}$.  Since $\mathcal E_1^{(2)}$ is an $I_0$ set (as it is independent), this special case of Lemma \ref{lem:EpContinuity} is also a special case of Proposition \ref{prop:RelativeKatznelson2Diff}.
\end{remark}

As with Conjecture \ref{conj:SpecialKatznelson4Z}, we are unable to prove any nontrivial instance of Conjecture \ref{conj:KatznelsonEp}.  One special case  of this conjecture in $\mathbb F_3^\omega$ is given by $S_{3AP}:=\{\mb e_{n}+\mb e_{n+d}+\mb e_{n+2d}:n, d\in \mathbb N\}$.  The collection $\{\{n,n+d,n+2d\}:n,d\in \mathbb N\}$ is partition regular, by van der Waerden's theorem on arithmetic progressions, so Lemma \ref{lem:EpContinuity} says that $S_{3AP}$ is Bohr recurrent.  The problem of proving (or refuting) that $S_{3AP}$ is a set of chromatic recurrence seems interesting, but it is unclear whether this special case will be any easier than the full generality of Conjecture \ref{conj:KatznelsonEp}.

\begin{remark}
Like the proof of Proposition \ref{prop:RelativeKatznelson4Diff}, our proof of Proposition \ref{prop:KEpImpliesKFp} uses only those properties of chromatic recurrence listed in Remark \ref{rem:Real}.  Trivial modifications of the proof of Proposition \ref{prop:KEpImpliesKFp} therefore prove the following generalization.
\end{remark}

\begin{proposition}
  Let $p$ be prime, let $d>2$ be divisible by $p$, and let $\mathcal R$ be a collection of subsets of $\mathbb F_p^\omega$ such that
  \begin{enumerate}
    \item[(i)]  if $\rho:\mathbb F_p^\omega\to\mathbb F_p^\omega$ is a surjective homomorphism and $R\in \mathcal R$, then $\rho(R)\in \mathcal R$;

    \item[(ii)]  if $S\in\mathcal R$ and $S'\supseteq S$, then $S'\in \mathcal R$.
  \end{enumerate}
If every subset of $\mathcal E_d^{(p)}$ which is a set of Bohr recurrence belongs to $\mathcal R$, then every subset of $\mathbb F_p^\omega$ which is a set of Bohr recurrence also belongs to $\mathcal R$.
\end{proposition}

\subsection{Cayley graphs}

If $G$ is an abelian group and $V, S\subseteq  G$, the \emph{Cayley graph} $\Cay(V,S)$ has vertex set $V$, with two vertices $g,g'$ joined by an edge if $g-g'\in S$ or $g'-g\in S$. In this notation, the usual Cayley graph determined by $S$ is $\Cay(G,S)$. It is easy to see that $S$ is a set of chromatic recurrence if and only if the chromatic number of $\Cay(G,S)$ is infinite - see \cite{Katznelson} for further exposition.

Combining Proposition \ref{prop:KEpImpliesKFp} and Lemma \ref{lem:EpContinuity}, we see that the following conjecture, for a fixed $p$, is equivalent to Conjecture \ref{conj:KatznelsonFp}, for the same $p$.  To state it, we define the \emph{chromatic number} of a hypergraph with vertex set $V$ and edge set $\mathcal F$ to be the minimum $k\in \mathbb N$ such that $V$ may be partitioned into $k$ sets where no $F\in \mathcal F$ lies in one cell of the partition.  If no such $k$ exists, the hypergraph has infinite chromatic number.

\begin{conjecture}
  Let $p$ be an odd prime.  If $\mathcal F$ is the edge set of a $p$-uniform hypergraph with vertex set $\mathbb N$ having infinite chromatic number, then $\Cay(\mathbb F_p^\omega,\mathcal E(\mathcal F))$  has infinite chromatic number.
\end{conjecture}

\section{Why Conjectures \ref{conj:SpecialKatznelson4Z} and \ref{conj:KatznelsonEp} are not as easy as Proposition \ref{prop:RelativeKatznelson2Diff}}\label{sec:DashHopes}

We now explain why Conjecture \ref{conj:KatznelsonEp} (when $d>2$) is not susceptible to the method of proof of Proposition \ref{prop:RelativeKatznelson2Diff}.  We first explain the latter proof in terms of Cayley graphs.

The proof of Proposition \ref{prop:RelativeKatznelson2Diff} showed that when $V$ is an $I_0$ set and $S\subseteq  V-V$ is a set of Bohr recurrence, then $\Cay(V,S)$ has infinite chromatic number.  Now consider the special case of Conjecture \ref{conj:KatznelsonEp} with $p=2$ and $d=4$.  One may attempt to prove this special case analogously to Proposition \ref{prop:RelativeKatznelson2Diff}, by fixing a set of Bohr recurrence $S\subseteq  \mathcal E_4^{(2)}\subseteq \mathbb F_2^\omega$, setting
$V=\mathcal E_2^{(2)}$, and proving that $\Cay(V,S)$ has infinite chromatic number. But this $\Cay(V,S)$ may have finite chromatic number, as we demonstrate below.  Since $\Cay(V,S)$ is not all of $\Cay(\mathbb F_2^{\omega},S)$, this example will not disprove Conjecture \ref{conj:KatznelsonEp}.

With $V=\mathcal E_2^{(2)}$, we will find a set of Bohr recurrence $S\subseteq \mathcal E_4^{(2)}$ such that $\Cay(V,S)$ has chromatic number $2$.  Instead of the usual presentation of $\mathbb F_2^\omega$, we consider $Fin(\mathbb Z\times \mathbb Z)$, the collection of finite subsets of $\mathbb Z\times \mathbb Z$, with the group operation of symmetric difference.  Since $Fin(\mathbb Z\times \mathbb Z)$ is a countably infinite abelian group all of whose elements have order $2$, it is isomorphic to $\mathbb F_2^\omega$. Let
\[S_{\square}:=\{\{(n,m),(n+d,m),(n,m+d),(n+d,m+d)\}:n,m\in \mathbb Z, d\in \mathbb N\}\subseteq  Fin(\mathbb Z\times \mathbb Z),\]
so that $S_{\square}$ is the collection of subsets of the integer lattice forming the vertices of a square with sides parallel to the coordinate axes.  Gallai's multidimensional generalization of van der Waerden's theorem (see  Chapter 42 of \cite{ColoringBook} for exposition) implies $S_{\square}$ is partition regular.    By a straightforward modification of Lemma \ref{lem:EpContinuity}, $S_{\square}$ is therefore a Bohr recurrent subset of $Fin(\mathbb Z\times \mathbb Z)$.

Let $Fin_2$ be the $2$-elements subsets of $\mathbb Z\times \mathbb Z$, and let $\mathcal G:=\Cay(Fin_2,S_{\square})$.  This is the graph whose vertices are the $2$-element subsets of $\mathbb Z\times \mathbb Z$, where two vertices $\{a,b\}$, $\{c,d\}$ are joined by an edge if their symmetric difference forms the vertices of a square with sides parallel to the coordinate axes.   We will prove that $\mathcal G$ has chromatic number $2$, by showing that its connected components are each isomorphic to a either a singleton, a pair of vertices connected by an edge, or an infinite path. To see this, fix a two element set $v=\{(n,m),(r,s)\}\subseteq  \mathbb Z\times \mathbb Z$, and  note that:
\begin{enumerate}
\item[(i)] If $v$ is not two of the four vertices of some square parallel to the coordinate axes, then $v$ is not connected to any other element of $Fin_2$ by an edge of $\mathcal G$.

\item[(ii)] If $v$ forms two opposed vertices of a square, such as $\{(n,m),(n+d,m+d)\}$, then $v$ is connected in $\mathcal G$ to exactly one element $v'$ of $Fin_2$, consisting of the other two vertices of the square, and this $v'$ is connected only to $v$.

\item[(iii)] If $v$ forms two vertices on the same side of a square, such as $\{(n,m),(n,s)\}$, then $v$ lies on the infinite path $P$ in $\mathcal G$ with vertices $\{(n+jd,m),(n+jd,s):j\in \mathbb Z\}$, where $d=m-s$.  Note that each element in $P$ is connected to exactly two other elements of $P$.  It is easy to check that $P$ contains no cycle, so it is indeed an infinite path.
\end{enumerate}
Thus the connected components of $\mathcal G$ have chromatic number $2$, so $\mathcal G$ itself has chromatic number $2$.  To interpret this example in the usual presentation of $\mathbb F_2^\omega$, we define an isomorphism $\phi$ between $\mathbb F_2^\omega$ and $Fin(\mathbb Z\times \mathbb Z)$.  First fix a bijection $b:\mathbb N\to \mathbb N\times \mathbb N$.  Now define $\phi: F_2^\omega\to Fin(\mathbb Z\times \mathbb Z)$ by $\phi(x)=b(\{i\in \mathbb N:x_i\neq 0\})$.  This isomorphism identifies $\mathcal E_2^{(2)}$ with $Fin_2$, and likewise identifies $S_{\square}$ with a subset of $\mathcal E_4^{(2)}$.

In light of the above, it may be worthwhile to investigate $S_{\square}$ as a potential counterexample to Conjecture \ref{conj:KatznelsonFp}.

\begin{remark}
  A construction similar to the one above will demonstrate why Conjecture \ref{conj:SpecialKatznelson4Z} cannot be proved using our proof of Proposition \ref{prop:RelativeKatznelson2Diff}.
\end{remark}

\section{Suggestions}\label{sec:Suggestions}

Propositions \ref{prop:RelativeKatznelson2Diff}, \ref{prop:KrizInDiff}, and Theorem 1.3 of \cite{griesmer2020separating} suggest the following problems.  For the first, it is useful to note that a set $S\subseteq  \mathbb Z$ is dense in the Bohr topology of $\mathbb Z$ if and only if for all $m\in \mathbb Z$, the translate $S-m$ is a set of Bohr recurrence.

\begin{problem}\label{prob:DenseVersion}
  Prove or disprove: there is a set $S\subseteq  \mathbb Z$ which is dense in the Bohr topology of $\mathbb Z$ such that if $S'\subseteq  S$ is also dense in the Bohr topology, then for all $m\in \mathbb Z$, $S' - m$ is a set of chromatic recurrence.
\end{problem}

\begin{problem}\label{prob:HereditarySeparation}Prove or disprove: if $S\subseteq  \mathbb Z$ is a set of density recurrence, then there is an $S'\subseteq S$ which is a set of chromatic recurrence and not a set of density recurrence.
\end{problem}

The variants of these problems where $\mathbb Z$ is replaced by another countable abelian group are equally interesting.

\section{How Theorem \ref{thm:Bogoliouboff} follows from F{\o}lner's proof}\label{sec:Folner}

Here here are some remarks on \cite{Folner} to briefly indicate how the proof of \cite[Theorem 1]{Folner} implies Theorem \ref{thm:Bogoliouboff} above.

Suppose $G$ is an infinite discrete abelian group.  Given a function $f$ on $G$ and $t\in G$, we write $f_t$ for the translate of $f$ defined by $f_t(x):= f(x-t)$. An \emph{invariant mean} (or \emph{Banach mean value} in the terminology of \cite{Folner}) is a positive linear functional on $l^\infty(G)$ satisfying $M(1_G)=1$ and $M(f_t)=M(f)$ for every $f\in l^\infty(G)$ and $t\in G$. When displaying the variable of $f$, we may write $M_t(f(t))$ for $M(f)$.

The proof of Theorem 1 of \cite{Folner}, mostly contained in Section 4 therein, begins by fixing an invariant mean $M$ on $l^\infty(G)$.  Given a partition  $G=A_1\cup \dots \cup A_r$ into finitely many sets, we have $M(1_{A_i})>0$ for at least one $i$. We now fix such an $i$ and omit the subscript, writing $A$ for this $A_i$. Based on this $M$, a function $\mu$ on $G$ is constructed with the property that $\mu(x)>0$ implies that $A\cap (A-x)$ ($=\{a\in A: \{a,a+x\}\subset A\}$) satisfies $M(1_{A\cap (A-x)})>0$. Note that when $S\subset G$ is finite, we have $M(1_S)=0$, so $M(1_S)>0$ implies $S$ is infinite.  Thus $A\cap (A-x)$ is infinite for each $x\in G$ where $\mu(x)>0$.

Finally, the function $\mu*\mu:G\to \mathbb R$ is defined by $\mu*\mu(x):=M_t(\mu(t) \mu(x-t))$.  More formally, one could write $\mu*\mu(x):=M(\mu \cdot \tilde{\mu}_x)$, where $\tilde{\mu}(t):=\mu(-t)$.

\begin{observation}\label{obs:Convolution}  If $x\in G$ satisfies $\mu*\mu(x)>0$, then $x\in \Delta_2(A)$.
\end{observation}
To see this, note that if $\mu*\mu(x)>0$, then the set of $t$ where $\mu(t)\mu(x-t)>0$ must be infinite.  For such $t$, we have $\mu(t)>0$ and $\mu(-t+x)>0$, so there are infinitely many $a\in A$ such that $\{a,a+t\}\subset A$, and infinitely many $a'\in A$ such that $\{a',a'+x-t\}\subset A$.  For each $x\in G$ with $\mu*\mu(x)>0$, we may therefore find a $t\in G$ and $a, a'\in A$ so that $a, a+t, a',$ and $a'+x-t$ are mutually distinct.  Setting $a_1=a,a_2=a+t,a_3 = a'$, and $a_4=a'+x-t$, we have $(a_1-a_2) - (a_3-a_4) = x$.  So every $x$ where $\mu*\mu (x)>0$ can be written as $(a_1-a_2)-(a_3-a_4)$, with the $a_i\in A$ mutually distinct.  This means that $\{x:\mu*\mu(x)>0\} \subseteq \Delta_2(A)$.

In Section 4 of \cite{Folner} it is shown that $\{x\in G: \mu*\mu(x)>0\}$ contains a Bohr neighborhood of $0$.  Combined Observation \ref{obs:Convolution}, we see that $\Delta_2(A)$ contains a Bohr neighborhood of $0$.

\frenchspacing
\bibliographystyle{amsplain}
\bibliography{KProblemRemarks}

\end{document}